\documentclass{amsart} \usepackage{url}
\usepackage{amsmath,amsfonts,euscript,amsthm,amssymb,amscd}
\usepackage[T1]{fontenc}
\usepackage{lmodern}
\usepackage{epigraph}
\usepackage[utf8]{inputenc}
\usepackage{enumerate}
\input{diag}

\theoremstyle{definition}
\newtheorem{Def}{Definition}[subsection]
\newtheorem{Def-Prop}[Def]{Definition-Proposition}

\newtheorem{Ex}[Def]{Example}
\newtheorem{Th}[Def]{Theorem}

\newtheorem{remark}[Def]{Remark}
\newtheorem{Prop}[Def]{Proposition}
\newtheorem{Lemma}[Def]{Lemma}
\newtheorem{Cor}[Def]{Corollary}

\newtheorem{Hyp}{Hypothesis}

\newcommand{\hdot}{{\:\raisebox{3pt}{\text{\circle*{1.5}}}}}

 \newcommand{\Ba}{\begin{array}}
 \newcommand{\Ea}{\end{array}}
\usepackage[all]{xy}
\usepackage{tikz}
 \tikzstyle{int}=[circle, draw,fill=black,outer sep=0,minimum size=3pt, inner sep=0]
  \tikzstyle{ext}=[circle, draw=black,outer sep=0,inner sep=1pt]

\begin{document}

\author{Alexey Kalugin} 
\address{Max Planck Institut für Mathematik in den Naturwissenschaften, Inselstraße 22, 04103 Leipzig, Germany}
\email{alexey.kalugin@mis.mpg.de}
\title{Remarks on the Grothendieck-Teichmüller group and A. Beilinson's gluing.}

\maketitle
\begin{abstract} In this note, we study A. Beilinson's gluing for perverse sheaves in the case of the diagonal arrangement and its relation to the Grothendieck-Teichmüller group. We also explain a relation to the Etingof-Kazhdan quantisation. 

\end{abstract}

\subsection{Introduction} Let $\textsf M^{\mathrm B}(\mathbb A^n,\mathcal S_{\emptyset})$ be a category of unipotent perverse sheaves on a complex $n$-affine space which are lisse with respect to a diagonal stratification $\mathcal S_{\emptyset}.$ We propose the following: 

\begin{Hyp}\label{Hyp1}
\begin{enumerate}[(i)]
\item For every binary $n$-labelled tree $T$ there exist a fiber functor:
$$
\omega_T\colon \textsf M^{\mathrm B}(\mathbb A^n,\mathcal S_{\emptyset})\longrightarrow \textsf {Vect}_{\mathbb Q}
$$
\item A collection $\mathcal Loc_n:=\{ \textsf M^{\mathrm B}(\mathbb A^n,\mathcal S_{\emptyset}),\omega_T\}_{T\in Tree(n)}$ naturally assembles into a fibered category over a category $\Pi_1^{\mathrm B}({FM}^n(\mathbb A)),$ where $\Pi_1^{\mathrm B}({FM}^n(\mathbb A)),$ is a Betti (=pro-unipotent) fundamental groupoid of the Fulton-MacPherson space of $n$-points in $\mathbb A$ \cite{FM}. 
\item The corresponding category of cartesion sections gives a $\Sigma_n$-equivariant equivalence:
$$
\Gamma_{cart}(\mathcal {L}oc_n)\cong \textsf M^{\mathrm B}(\mathbb A^n,\mathcal S_{\emptyset}),
$$
where a symmetric group $\Sigma_n$ acts on $\mathbb A^n$ by permuting coordinates. 
\end{enumerate}
\end{Hyp}  
Note that a collection $\{\Pi_1^{\mathrm B}({FM}^n(\mathbb A))\}_{n \geq 1}$ has a natural structure of an operad \cite{Kon4}. The equivalences from Hypothesis \ref{Hyp1} are compatible with operadic compositions. Denote by $\mathrm {GT}_{\mathrm {un}}$ the pro-unipotent Grothendieck-Teichmüller group \cite{Drin3}. By $\textsf {M}^{\mathrm {B}}(Ran(\mathbb A),\mathcal S_{\emptyset})$ we denote a category of unipotent perverse sheaves on a Ran space of $\mathbb A^1$ \cite{BD} \cite{KalQ}. These lead to the following: 

\begin{Hyp}\label{Hyp2} There exists a morphism:
$$ 
\mathrm {GT}_{\mathrm {un}}\longrightarrow \mathrm {Aut}(\textsf {M}^{\mathrm {B}}(Ran(\mathbb A),\mathcal S_{\emptyset})
$$

\end{Hyp} 
Under the equivalence between factorizable objects in $\textsf {M}^{\mathrm {B}}(Ran(\mathbb A),\mathcal S_{\emptyset})$ and conilpotent Hopf algebras \cite{KS2} \cite{KalQ} Hypothesis \ref{Hyp2} corresponds to Theorem $11.1.7$ from \cite{Fre} (there is a natural equivalence between $2$-algebras (DG-algebras over and operad of little $2$-disks) and DG-sheaves on a Ran space \cite{Lu}. Following V. Schechtman \cite{Sch1} we consider examples of above statements in the case of affine spaces $\mathbb A^2$ and $\mathbb A^3$ and discuss a relation to the Etingof-Kazhdan quantisation \cite{EK}. A more detailed account shall appear in \cite{KalT}. 
\subsection{Acknowledgments} This work was supported by the Max Planck Institute for Mathematics in Sciences.

\subsection{Notation} For an integer $n$ we denote by $[n]$ a set of elements $[n]:=\{1,2,\dots,n\}.$ We denote by $\Sigma_n$ the symmetric group on $n$ letters. Let $\textsf C$ be an abelian category by a fiber functor $\omega\colon \textsf A\rightarrow \textsf {Vect}_{\mathbb Q}$ we understand an exact and faithful functor \cite{Deligne}, where $\textsf {Vect}_{\mathbb Q}$ is a category of finite dimensional $\mathbb Q$-vector spaces. By $\textsf C_{\omega}^{\textsf A}$ we denote the corresponding Tannakian dual coalgebra. By $\textsf {Cat}$ we denote a $2$-category of all small categories.

\subsection{A. Beilinson's gluing} Let $(X,\mathcal O_X)$ be a complex variety space equipped with a Whitney stratification $\mathcal S.$ By $\textsf M(X,\mathcal S)$ we denote a category of perverse sheaves smooth with respect to $\mathcal S.$\footnote{Here we assume the middle perversity function in the sense of \cite{BBD}.} With every regular function $f\colon X\rightarrow \mathbb A^1$ we can associate the following diagram of algebraic varieties $i\colon D\longrightarrow X\longleftarrow U\colon j. $ Here by $D$ we have denoted the principal divisor defined by $D:=f^{-1}(0)$ and by $U:=f^{-1}(\mathbb C^{\times})$ the corresponding open complement. Following P. Deligne \cite{DELC}\footnote{We shift cycles by $[-1]$ in order to make them $t$-exact.} and O. Gabber \cite{BBD} we have a {functor of nearby cycles} $\Psi_{f}\colon \textsf {M}(X,\mathcal S)\longrightarrow \textsf {M}(Z,\mathcal S)$ and a functor of vanishing cycles $\Phi_{f}\colon \textsf{M}(X,\mathcal S)\longrightarrow \textsf{M}(Z,\mathcal S).$ We have a natural transformation $T_{\Psi}\colon\Psi_f\longrightarrow   \Psi_f$ (resp. $T_{\Phi}\colon\Phi_f\longrightarrow   \Phi_f$) called the monodromy transformation of nearby cycles (resp.  {monodromy transformation of vanishing cycles}.) We also have {canonical and variations morphism}, which are natural transformation of functors: $can\colon \Psi_{f}\longleftrightarrow\Phi_ f\colon var.$ We also denote by $\Psi_f^u$ (resp. $\Phi_f^u$) a part of nearby (resp. vanishing) cycles where a monodromy operator act unipotently. Following A. Beilinson \cite{B2} with a regular function $f$ on $X$ we associate a {gluing category} $\textsf{Glue}_f(U,Z).$ This is a category with following objects $
\{\mathcal E_U,\mathcal E_Z,u,v\},$ where $u\colon \Psi_f^u\mathcal E_U\longrightarrow \mathcal E_Z,\quad v\colon \Psi_{f}^{u}\mathcal E_U\longleftarrow \mathcal E_Z,$ 
where $\EuScript E_U \in \textsf M(U,\mathcal S)$ and $\EuScript E_Z \in \textsf M(Z,\mathcal S),$ such that $vu=T_{\Psi}-1.$ We have the following: 
\begin{Th}[A. Beilinson]\label{Bel1} For every $f\in \mathcal O_X$ we have a functor $F_f:$ 
\begin{equation}\label{}
F_f\colon \textsf M(X,\mathcal S)\longrightarrow \textsf {Glue}_f(U,Z)$$
defined by the rule:
$$F_f\colon \mathcal E\longmapsto \{j^*\mathcal E,\Phi_{f}^{u}\mathcal E,can,var\}
\end{equation}
This functor extends to an equivalence between categories $\textsf M(X,\mathcal S)$ and $\textsf{Glue}_f(U,Z).$
\end{Th}

\subsection{Fiber functors and trees}  For a natural number $n\in \mathbb N_{\geq 1}$ we consider the corresponding complex affine space $\mathbb A^n$ with coordinates $(z_i)_{i=1,\dots n}.$ We equip $\mathbb A^n$ with a {diagonal stratification} $\mathcal S_{\emptyset}=\{\Delta_{ij}\},$ where $\Delta_{ij}=z_i-z_j.$ The unique minimal closed stratum of $\mathcal S_{\emptyset}$ will be denoted by $\Delta$ and the unique maximal open stratum will be denoted by $U^{n}.$ We denote by $\textsf M^{\mathrm B}(\mathbb A^n,\mathcal S_{\emptyset})$ a category of perverse sheaves which are smooth with respect to the diagonal stratification $\mathcal S_{\emptyset}$ and every perverse sheaf is an extension of direct sums of perverse sheaves supported on closed strata of $\mathbb A^n$ \cite{Kho}. Denote by $Tree(n)$ a set (groupoid) of binary rooted trees with leaves labelled by a finite set $[n].$ We are going to define fiber functors associated with a tree $T\in Tree(n):$
\begin{Ex} We start with the simplest (nontrivial) case $\mathbb A^2$ with coordinates $(z_1,z_2).$ There two binary $2$-labelled trees:
$$
T_1=\begin{xy}
 <0mm,-0.55mm>*{};<0mm,-2.5mm>*{}**@{-},
 <0.5mm,0.5mm>*{};<2.2mm,2.2mm>*{}**@{-},
 <-0.48mm,0.48mm>*{};<-2.2mm,2.2mm>*{}**@{-},
 <0mm,0mm>*{\circ};<0mm,0mm>*{}**@{},
 <0.5mm,0.5mm>*{};<2.7mm,2.8mm>*{^{1}}**@{},
 <-0.48mm,0.48mm>*{};<-2.7mm,2.8mm>*{^{2}}**@{},
 \end{xy}\qquad T_2=\begin{xy}
 <0mm,-0.55mm>*{};<0mm,-2.5mm>*{}**@{-},
 <0.5mm,0.5mm>*{};<2.2mm,2.2mm>*{}**@{-},
 <-0.48mm,0.48mm>*{};<-2.2mm,2.2mm>*{}**@{-},
 <0mm,0mm>*{\circ};<0mm,0mm>*{}**@{},
 <0.5mm,0.5mm>*{};<2.7mm,2.8mm>*{^{2}}**@{},
 <-0.48mm,0.48mm>*{};<-2.7mm,2.8mm>*{^{1}}**@{},
 \end{xy}\
$$
We define functors $\omega_{T_{i}}\colon \textsf M^{\mathrm B}(\mathbb A^2,\mathcal S_{\emptyset})\longrightarrow \textsf {Vect}_{\mathbb Q}$ $i=1,2$ by the rule:
$$\omega_{T_1}:=\Gamma(\mathbb A,\Psi^{u}_{z_{1}-z_{2}}\oplus \Phi^{u}_{z_1-z_2})[-1],\qquad \omega_{T_2}:=\Gamma(\mathbb A,\Psi^u_{z_{2}-z_{1}}\oplus \Phi^u_{z_2-z_1})[-1].$$
By the construction these functors are exact and moreover by Theorem \ref{Bel1} implies that this functor is faithful and hence it is a fiber functor. Indeed we get the classical $(\Psi,\Phi)$-description of the category of perverse sheaves: we have a morphism $f_{T_1}\colon \mathbb A^2\rightarrow \mathbb A$ (resp. $f_{T_2}\colon \mathbb A^2\rightarrow \mathbb A$)defined by the rule $(z_1,z_2)\mapsto (z_1-z_2)$ (resp. $(z_1,z_2)\mapsto (z_2-z_1)$). The shifted pushforward defines an equivalence between $\textsf M^{\mathrm B}(\mathbb A^2,\mathcal S_{\emptyset})$ and category of perverse sheaves on $\mathbb A,$ which are smooth with respect to a stratification $\{0\}\subset \mathbb A.$ The corresponding Tannakian dual coalgebra is a quiver coalgebra of the quiver $\bullet\longleftrightarrow \bullet.$

\end{Ex}

\begin{Ex}\label{Ma} Consider a case of $\mathbb A^3$ with a coordinate $(z_1,z_2,z_3).$ For example we take the following tree: 
$$
T=
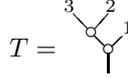
\begin{xy}
 <0mm,0mm>*{\circ};<0mm,0mm>*{}**@{},
 <0mm,-0.49mm>*{};<0mm,-3.0mm>*{}**@{-},
 <0.49mm,0.49mm>*{};<1.9mm,1.9mm>*{}**@{-},
 <-0.5mm,0.5mm>*{};<-1.9mm,1.9mm>*{}**@{-},
 <-2.3mm,2.3mm>*{\circ};<-2.3mm,2.3mm>*{}**@{},
 <-1.8mm,2.8mm>*{};<0mm,4.9mm>*{}**@{-},
 <-2.8mm,2.9mm>*{};<-4.6mm,4.9mm>*{}**@{-},
   <0.49mm,0.49mm>*{};<2.7mm,2.3mm>*{^1}**@{},
   <-1.8mm,2.8mm>*{};<0.4mm,5.3mm>*{^2}**@{},
   <-2.8mm,2.9mm>*{};<-5.1mm,5.3mm>*{^3}**@{},
 \end{xy}
$$
Let $\mathbb A^2$ with a coordinate $(t_1,t_2).$ Consider the morphism $f_T\colon \mathbb A^3\longrightarrow \mathbb A^2,$ defined by the rule $f_T\colon (z_1,z_2,z_3)\longmapsto (z_3-z_2,z_1-z_2).$ Denote by $\mathcal S$ stratification on $\mathbb A^2$ associated with hyperplanes $t_1=t_2,$ $t_1=0,$ $t_2=0.$ The morphism $f_T$ respect these stratification and defines an equivalence of abelian categories $f_{T*}[-1]\colon \textsf M^{\mathrm B}(\mathbb A^3,\mathcal S_{\emptyset})\overset{\sim}{\longrightarrow}  \textsf M^{\mathrm B}(\mathbb A^2,\mathcal S_{}).$ Consider the following quiver (we assume that $u$-morphisms go down and $v$-morphisms go up):
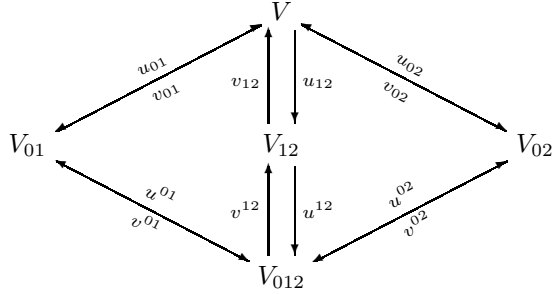
\begin{equation*}
\begin{diagram}[height=2.5em,width=4.8em]
 &  & V  & & \\
  & \ruTo_{v_{01}}\ldTo^{u_{01}} & \uTo^{v_{12}}\dTo_{u_{12}} & \luTo_{v_{02}}\rdTo^{u_{02}} &\\
V_{01} &  & V_{12} &   & V_{02} \\
  & \luTo_{v^{01}}\rdTo^{u^{01}} & \uTo^{v^{12}}\dTo_{u^{12}} & \ruTo_{v^{02}}\ldTo^{u^{02}} &\\
  &  & V_{012}  & & \\
\end{diagram}
\end{equation*}
where $V_{ij}$ and $V_{012}$ and $V$ are vector spaces. Building on Theorem \ref{Bel1} in \cite{Sch1} (see also \cite{Sch2} \cite{Sch3}) it was proved that the datum of the quiver together with some relations (see \textit{ibid}.) determines a perverse sheaf in $\textsf M^{\mathrm B}(\mathbb A^2,\mathcal S_{})$ and vive versa. Applying the equivalence above (here we use an interaction property of nearby and vanishing cycles for a pushforward along a proper morphism i.e. $\Psi_fg_*\cong g_*\Psi_{gf}$) one defines a fiber functor $\omega_T\colon  \textsf M^{\mathrm B}(\mathbb A^3,\mathcal S_{\emptyset})\longrightarrow \textsf {Vect}_{\mathbb Q}$ by the rule:
$$
\omega_T:=\Gamma(\mathbb A,\underbrace{\Psi^u_{z_1-z_2}\Psi^u_{z_3-z_2}}_{V}\oplus \underbrace{\Phi^u_{z_1-z_2}\Psi^u_{z_3-z_2}}_{V_{12}\oplus V_{02}}\oplus \underbrace{\Psi^u_{z_1-z_2}\Phi^u_{z_3-z_2}}_{V_{01}}\oplus \underbrace{\Phi^u_{z_1-z_2}\Phi^u_{z_3-z_2}}_{V_{012}})[-1]
$$
Analogously one defines a fiber functor for any tree $T\in Tree(3).$


\end{Ex}

\begin{remark} \begin{enumerate}[(i)] \item One extends the definition above to an arbitrary dimension. Let $T\in Tree(n),$ we define $\omega_T\colon \textsf M^{\mathrm B}(\mathbb A^n,\mathcal S_{\emptyset})\longrightarrow \textsf {Vect}_{\mathbb Q}$ by the rule:
\begin{equation}\label{bf}
\omega_T:=\bigoplus_{\Lambda=\Psi^u,\Phi^u} \Gamma(\mathbb A,\Lambda_{z_{i_{n}}-z_{i_{n-1}}}\oplus \dots\oplus \Lambda_{z_{i_1}-z_{i_2}})[-1]
\end{equation} 
where $(i_1,i_2)$ is pair of leaves which collide in the tree $T$ first (we orient a tree towards a root) as the second pair we take leaves $(i_3,i_2)$ which collide next (here we assume that $i_3$ is the closest leave to $i_2$. Note that if two pairs collide at the same we do not distinguish the order in the composition, indeed by Lemma $10.2$ \cite{BFS} two compositions are identically equal. Hence \eqref{bf} is well defined. 
\item Recall that Theorem \ref{Bel1} holds for $\EuScript D$-modules and more generally mixed Hodge modules \cite{Saito}. Hence fiber functors \eqref{bf} can be defined in the mixed Hodge (more generally motivic) setting. It would be very interesting to define and study ${\textsf C}^{ \textsf M^{\mathrm B}(\mathbb A^n,\mathcal S_{\emptyset})}_{\omega_T}$ as an object of the category of mixed Hodge structures. Note that this coalgebra is closely related to universal enveloping algebra of P. Deligne's motivic fundamental group \cite{DelF} at tangential base points. 
\end{enumerate} 
\end{remark}

\subsection{Local systems of categories} Further we assume that $n=1,2,3.$ Recall that a Fulton-MacPherson compactification $FM^n(\mathbb A^1)$ is defined as a real blowup of the space of $n$ distinct complex points \cite{FM}. This space is naturally a manifold with corners such that its interior can be identified with $U^n$ (modulo affine transformation). We consider a Betti fundamental groupoid (pro-unipotent completion of the Poincaré groupod) $\Pi_1^{\mathrm {B}}(FM^n(\mathbb A^1))$ with base points defined by points in the real strata of the smallest dimension. Such base points can be identified with $[n]$-labelled binary trees. We define a $2$-functor $\mathcal Loc_n\colon  \Pi_1^{\mathrm {B}}(FM^n(\mathbb A^1)\longrightarrow \textsf {Cat}$ by the rule: $T\longmapsto (\textsf {M}^{\mathrm B}(\mathbb A^n,\mathcal S_{\emptyset}),\omega_T)$ For a path $\gamma$ in $FM^n(\mathbb A^1)$ between two binary trees $T_i$ and $T_j$ we define an equivalence between categories with fiber functors as $\sigma_{T_i \,T_j}^*,$ where $\sigma_{T_i\, T_j}\colon \mathbb A^n\rightarrow \mathbb A^n$ is a unique permutation of coordinates such that $f_{T_i}\sigma_{T_i\,T_j}=f_{T_j}.$ One computes that the resulting operators acts unipotently and hence we get a representation of the pro-unipotent completion. Denote by $\Gamma_{cart}(\mathcal L oc_n)$ the category of cartesion section of the corresponding fibration in the sense of A. Grothendieck \cite{SGA1}. We have the following:
\begin{Prop}\label{equi} We have a $\Sigma$-equivariant equivalence of categories:
$$
\Gamma_{cart}(\mathcal L oc_2)\cong \textsf M^{\mathrm B}(\mathbb A^2,\mathcal S_{\emptyset}),\quad \Gamma_{cart}(\mathcal L oc_3)\cong \textsf M^{\mathrm B}(\mathbb A^3,\mathcal S_{\emptyset}),
$$
\end{Prop} 
\begin{proof} We leave it to the reader, however see \cite{KaSh} (Subsection 9A) for $n=2.$

\end{proof} 
Denote by $\EuScript G_n$ a group of automorphisms of $\Pi_1^{\mathrm {B}}(FM^n(\mathbb A^1))$ which are identical on objects. From Proposition \ref{equi} one gets the following:
\begin{Cor} For $n=1,2,3$ we have a canonical action of a group $\EuScript G_n$ on $\textsf M^{\mathrm B}(\mathbb A^n,\mathcal S_{\emptyset})$). 

\end{Cor} 

\begin{remark} 

\begin{enumerate}[(i)]
\item Note that $\{\Pi_1^{\mathrm {B}}(FM^n(\mathbb A^1))\}_{n\geq 1}$ is naturally an operad in the category of groupoids. One shows that $\{\mathcal Loc\}$ is naturally a local system on the operad $\{\Pi_1^{\mathrm {B}}(FM^n(\mathbb A^1))\}_{n\geq 1}$ in the sense of \cite{GK} and the category of section is equivalent to the category of perverse sheaves on the Ran space.

\item Recall that the group of automorphisms of the operad $\{\Pi_1^{\mathrm {B}}(FM^n(\mathbb A^1))\}_{n\geq 1}$ is the Grothendieck-Teichmüller group $\mathrm {GT}_{\mathrm {un}}$ \cite{Fre}. Hence one proves Hypothesis \ref{Hyp2}. It would be very interesting to consider a "derived" version of this picture in particular to relate M. Kontsevich's graph complex to deformation of the category of $!$-sheaves on a Ran space of $\mathbb A^1.$ 
\end{enumerate} 

\end{remark} 

\subsection{Quantisations} In \cite{KalQ} the problem of quantisation of Lie bialgebras was transformed to the problem of constructing isomorphisms between certain fiber functors. Consider a fiber functor $\omega^B$ from \textit{ibid}. This functor is defined as the zero cohomology of the smallest real diagonal with coefficients in sections with real support (hyperbolic stalk) see \cite{KaSh}. We will discuss the space of isomorphisms between functors $\omega^B$ and $\omega_T.$\footnote{In \cite{FPS} the same problem was studied in the case of a normal crossing arrangement $z_1\dots z_n=0$.}
\par\medskip 
Let $\mathcal S_{\emptyset,\mathbb R}$ be a diagonal stratification of a real $n$-affine space $\mathbb A_{\mathbb R}^n,$ we assume that $(x_1,\dots, x_n)$ is a coordinates of the real affine space such that $\mathfrak R(z_i)=x_i.$ According to \textit{ibid.} a perverse sheaf is completely determined by a so-called hyperbolic sheaf i.e. a collection of vector spaces $E_C$ where $C\in \mathcal S_{\emptyset,\mathbb R}$ is a face and operators:  $\gamma_{C}^{C'}\colon E_{C}\rightarrow E_{C'}, \delta_{C'}^C\colon E_{C'}\rightarrow E_{C} $ when ${C}\subset \overline{C}'$ together with some relations (see \textit{ibid}.). The hyperbolic stalks are defined the rule $E_C:=\Gamma(C,\underline{\mathbf R\Gamma}_{\mathbb A_{\mathbb R}^n}).$ We usually denote chambers $C\in \mathcal S_{\emptyset,\mathbb R}$ as totally ordered real numbers i.e. $x_1<x_2<x_3,$ we also sometimes denote by $\Delta^{\mathbb R}$ the minimal diagonal. The following real-analytic interpretation of nearby and vanishing cycles will be important to us:
\begin{Lemma}[M. Kashiwara and P. Schapira \cite{KS}]\label{KS1} For every regular function $f\in \mathcal O_X$ we have the following isomorphism of functors:
$$
\Phi_f\overset{\sim}{\longrightarrow} i^*\underline{\mathbf R^{\hdot}\Gamma}_{\{\mathfrak{R}(f)\geqslant 0\}}[1]
$$
where $i\colon f^{-1}(0):=D\hookrightarrow X.$ 
\end{Lemma}
Let $X=\mathbb A^I$ we denote by $\Phi_f^{fake}$ (resp. $\Psi_f^{fake}$) the following functor $i^*\underline{\mathbf R^{\hdot}\Gamma}_{\{\mathfrak {R}(f)\geq 0\}}$  (resp. $i^*\underline{\mathbf R^{\hdot}\Gamma}_{\{\mathfrak {R}(f)<0\}}$) and called it a {fake vanishing (resp. nearby) cycles functor}. From the standard Gysin triangle we have the distinguished triangle:
\begin{equation}\label{im}
\Phi_f[-1]^{fake} \rightarrow i^*\underline{\mathbf R^{\hdot}\Gamma}_{\mathbb A^n_{\mathbb R}}\rightarrow \Psi_f^{fake} {\rightarrow} \Phi_f^{fake}
\end{equation} 
These functor are equipped with a natural transformations $\Psi_f^{fake}\rightarrow \Psi_f$ and $\Phi_f^{fake}\rightarrow \Phi_f$ (Lemma \ref{KS1}) which induce equivalences on the sections with support on a real locus and hence $\mathbf R^{\hdot}\Gamma(D,\Psi_f^{fake})\cong \mathbf R^{\hdot}\Gamma(D,\Psi_f)$ and $\mathbf R^{\hdot}\Gamma(D,\Phi_f^{fake})\cong \mathbf R^{\hdot}\Gamma(D,\Phi_f)$ \cite{FKS}. 

\begin{Ex} Consider the case $\mathbb A^2$ and a fiber functor $\omega_{T_1}.$ Following \cite{KaSh} Subsection $9$A we have:
$$
\Gamma(\mathbb A,(\Psi_{z_1-z_2}\oplus \Phi_{z_1-z_2})\cong \Gamma(\mathbb A,i^*\underline{\mathbf R^{\hdot}\Gamma}_{x_1<x_2}\oplus i^*\underline{\mathbf R^{\hdot}\Gamma}_{x_1\geq x_2})[1]
$$
Applying \eqref{im} we get a morphism from a functor $\omega^B$ to a functor $\omega_{T_1}.$ One shows (see \textit{ibid}.) that this is an equivalence.

\end{Ex}

\begin{Ex} Consider $\mathbb A^3$ with a binary tree $T$ from Example \ref{Ma}. Applying base change one easily computes that $V:=E_{x_1<x_3<x_2}.$ Namely denote by $\mathfrak R(z_3)<\mathfrak R(z_2)$ the locus $\mathbb A^3$ whcih consists of real numbers $(x_1,x_2,x_3)$ such that $x_3<x_2.$ Consider the following diagram:
\begin{diagram}[height=1.9em,width=3em]
 \mathfrak R(z_3)<\mathfrak R(z_2) & \rTo^{v} & \mathbb A^3_{\mathbb R} &   \rTo^{h}   &            \mathbb A^3 &  \\
  & &\uTo^{q} &      &         \uTo_{i_1}     \\
 \mathfrak R(z_3)=\mathfrak R(z_2) >\mathfrak R(z_1)  & \rTo^r & \mathbb A_{\mathbb R}^2 &   \rTo^{p}   &           \mathbb A^2_{z_3=z_2}\\
  & &\uTo^{l} &      &         \uTo_{i_2}     \\
    &  & \mathbb A_{\mathbb R}  &   \rTo^k  &           \mathbb A^1\\
    \end{diagram}
$\Psi_{z_3-z_2}^{fake}:=\mathbf R^{\hdot}i^*_1h_!v_*v^*h^!=\mathbf R^{\hdot} p_!q^*v_*v^*h^!$ and $\Psi_{z_1-z_2}^{fake}:=\mathbf R^{\hdot} i^*_2p_*r_*r^*p^!.$ Hence $\Psi_{z_1-z_2}^{fake}\Psi_{z_3-z_2}^{fake}=\mathbf R^{\hdot}k_*l^*r_*r^*q^*v_*v^*h^!.$ Note that $h^!$ is an exact functor (see \cite{KaSh}) (it takes perverse sheaves to combinatorial sheaves on the real affine space $\mathbb A^3_{\mathbb R}$). Hence it is enough to compute the $*$-extension of the corresponding combinatorial sheaves: let $\mathcal K=v^*h^!\mathcal E,$  $\mathcal E\in \textsf M^{\mathrm B}(\mathbb A^3,\mathcal S_{\emptyset}),$ be a combinatorial sheaf on $\mathfrak R(z_3)<\mathfrak R(z_2)$ (we denote the corresponding combinatorial data (section over faces) by $E_C,$ where $C\in S_{\emptyset,\mathbb R}.$) We are interested in sections of $\mathbf R^{\hdot}v_*\mathcal K$ over faces which have a non empty intersection with an image of $q.$ These chambers are $\{x_1<x_3=x_2\}$ and $\{x_1>x_3=x_2\}.$ Moreover since further we take a pullback along $r$ it is enough to consider $\{x_1<x_3=x_2\}.$ To compute $\Gamma(\{x_1<x_3=x_2\},\mathbf R^{\hdot}v_*\mathcal K)$ we need to take sections over chambers whose closure contains $\{x_1<x_3=x_2\}$ and have a non empty intersection with  $\mathfrak R(z_3)<\mathfrak R(z_2).$ This chamber is $\{x_1<x_3<x_2\},$ hence $V=E_{x_1<x_3<x_2}.$
\par\medskip 
\begin{diagram}[height=1.9em,width=3em]
&  &  \mathfrak R(z_3)\geq \mathfrak R(z_2) &   \rTo^{j}   &            \mathbb A^3 &  \\
  & &\uTo^{q} &      &         \uTo_{i_1}     \\
 \mathfrak R(z_3)=\mathfrak R(z_2) > \mathfrak R(z_1)  & \rTo^r & \mathbb A_{\mathbb R}^2 &   \rTo^{p}   &           \mathbb A^2_{z_3=z_2}\\
  & &\uTo^{l} &      &         \uTo_{i_2}     \\
    &  & \mathbb A_{\mathbb R}  &   \rTo^k  &           \mathbb A^1\\
    \end{diagram}
We have $\Phi_{z_3-z_2}^{fake}:=\mathbf R^{\hdot} i_1^*j_!j^!=\mathbf R^{\hdot} p_!q^*j^!$ and hence $\Psi_{z_1-z_2}^{fake}\Phi_{z_3-z_2}^{fake}=\mathbf R^{\hdot} k_*l^*r_*r^*q^*j^!.$ Let us compute section of $q^*j^!$ over a chamber $\{x_3=x_2>x_1\}.$ Let $\mathcal K$ be a combinatorial sheaf on $\mathbb A^3_{\mathbb R}$ (which is a $!$-restriction of a perverse sheaf) we need to compute its sections with support oi $\mathfrak R(z_3)\geq \mathfrak R(z_2)$ (we restrict ourselves to a chamber $\{x_3=x_2>x_1\}$). The section of $\mathbb D(\mathcal K)$ over chambers $\{x_3=x_2>1\}$ and $\{x_3>x_2>x_1\}$ are $E_{x_3>x_2>x_1}^*\oplus E_{x_2>x_3>x_1}^*\rightarrow E_{x_3=x_2>x_1}^*$ and $E_{x_3>x_2>x_1}^*.$ Hence section with support in $\mathfrak R(z_3)\geq \mathfrak R(z_2)$ over a chamber $x_3=x_2>x_1$ are given by the cohomology of the following complex
$$C^{\hdot}:=\{E_{x_3=x_2>x_1}\oplus E_{x_3>x_2>x_1}\overset{\gamma+\gamma+id}{\longrightarrow}E_{x_2>x_3>x_1}\oplus E_{x_3>x_2>x_1}\}.$$
 Since we are working with a perverse sheaf this complex has only cohomology in degree zero $H^0(C)= \mathrm {Ker}(E_{x_2=x_3>x_1}\rightarrow E_{x_2>x_3>x_1}).$ We set $V_{01}:= \mathrm {Ker}(E_{x_2=x_3>x_1}\rightarrow E_{x_2>x_3>x_1}).$
\par\smallskip 
Consider the following diamgram:
\begin{diagram}[height=1.9em,width=3em]
&  &  \mathfrak R(z_3)\geq \mathfrak R(z_2) &   \rTo^{j}   &            \mathbb A^3 &  \\
  & &\uTo^{q} &      &         \uTo_{i_1}     \\
 \mathfrak R(z_3)=\mathfrak R(z_1) \geq \mathfrak R(z_2)  & \rTo^r & \mathbb A_{\mathbb R}^2 &   \rTo^{t}   &           \mathbb A^2_{z_3=z_2}\\
  & \luTo^s & &      &         \uTo_{i_2}     \\
    &  & \mathbb A_{\mathbb R}  &   \rTo^k  &           \mathbb A^1\\
    \end{diagram}
By previous computations we have $\Phi_{z_1-z_2}^{fake}\Phi_{z_3-z_2}^{fake}:=\mathbf R^{\hdot} k_*s^*r^!q^*j^!.$  Let $\mathcal K$ be a combinatorial sheaf (again a $!$-restriction of a perverse sheaf) on $\mathbb A^3_{\mathbb R}$ we need to compute its sections with support on $\mathfrak R(z_3)\geq \mathfrak R(z_2)$ (we restrict ourselves to chamber $\Delta^{\mathbb R}$). To compute sections over $\Delta^{\mathbb R}$ we perform computation analogous to the previous case and find that it is equal:
$$A:=\mathrm {Ker}\left(E_{x_1=x_2=x_3}\overset{\oplus \gamma}{\longrightarrow} \bigoplus_{C\in \mathbb A^3_{\mathbb R}\setminus\mathfrak R(z_3)\geq \mathfrak R(z_2) \, \dim C=2} E_C\right).$$ 
Hence we set $V_{012}:=A.$

\par\smallskip
Consider the following diagram:
\begin{diagram}[height=1.9em,width=3em]
 \mathfrak R(z_3)<\mathfrak R(z_2) & \rTo^{d} & \mathbb A^3_{\mathbb R} &   \rTo^{j}   &            \mathbb A^3 &  \\
  & &\uTo^{q} &      &         \uTo_{i_1}     \\
 \mathfrak R(z_3)=\mathfrak R(z_1) \geq \mathfrak R(z_2)  & \rTo^r & \mathbb A_{\mathbb R}^2 &   \rTo^{p}   &           \mathbb A^2_{z_3=z_2}\\
  & \luTo^s & &      &         \uTo_{i_2}     \\
    &  & \mathbb A_{\mathbb R}  &   \rTo^k  &           \mathbb A^1\\
\end{diagram}
By previous computations we have $\Phi_{z_1-z_2}^{fake}\Psi_{z_3-z_2}^{fake}:=\mathbf R^{\hdot} k_*s^*r^!q^*d_*d^*j^!.$ In order to compute the iterated cycles we need to find sections of a sheaf $d_*d^*j^!\mathcal E$ over chambers $\{x_3=x_1>x_2\}$ $\{x_3=x_1<x_2\}$ and $\{x_1=x_2=x_3\}.$ Over over the first chamber are trivial since there are no chambers in $\mathfrak R(z_3)<\mathfrak R(z_2)$ such that their closure contains this chamber. Sections over the second chamber are given by the vector space $E_{x_3=x_1<x_2},$ since this chamber lie in the space $\mathfrak R(z_3)<\mathfrak R(z_2).$ Lets compute section over the minima chamber: we need to calculate the $*$-extension of the sheaf $d^*j^!\mathcal E,$ applying standard methods one get the following vector spaces:
$$E_{x_3<x_2=x_1},\quad E_{x_1=x_3<x_2},\quad E_{x_3<x_1=x_2}.
$$
Acting like before we finally set: $$V_{12}\oplus V_{02}:=E_{x_3<x_2=x_1}\oplus E_{x_3<x_1=x_2}.$$ Hence we have the following quiver:
\begin{equation*}
\begin{diagram}[height=2.5em,width=5em]
 &  & E_{x_1<x_3<x_2}  & & \\
  & \ruTo_{v_{01}}\ldTo^{u_{01}} & \uTo^{v_{12}}\dTo_{u_{12}} & \luTo_{v_{02}}\rdTo^{u_{02}} &\\
H^0(C) &  & E_{x_3<x_2=x_1} &   & E_{x_3<x_1=x_2} \\
  & \luTo_{v^{01}}\rdTo^{u^{01}} & \uTo^{v^{12}}\dTo_{u^{12}} & \ruTo_{v^{02}}\ldTo^{u^{02}} &\\
  &  & A  & & \\
\end{diagram}
\end{equation*}
We leave it to the reader to determine canonical and variation operators. 
\end{Ex}

\begin{remark} Recall that in \cite{KalQ} we consider "de Rham" fiber functor $\omega^{dR}$ for $\EuScript D$-modules. One can also construct an isomorphism between fiber functor $\omega^{dR}$ and the de Rham version of a functor $\omega_T.$ One can summarise by saying that there is a "canonical" Betti functor $\omega^B$ and a "canonical" de Rham functor $\omega^{dR}.$ A functor $\omega^B$ (resp. $\omega^{dR}$) is responsible for associative (resp. Lie) bialgebras and a problem of quantization \cite{Drin2} transfers to establishing an isomorphism between these two functors. The latter construction passes through "non-canonical" fiber functors $\omega_T.$ These functors have an advantage being "motivic" in contrast to functors $\omega^{dR}$ and $\omega^{B}.$
\end{remark} 

\bibliographystyle{amsalpha}
\bibliography{tt}

\end{document}